\documentclass[reqno,a4paper,12pt]{amsart}
\usepackage{latexsym, amsmath, amssymb, amsthm, a4}

\usepackage{amsfonts}

\usepackage[active]{srcltx}

\newtheorem{theorem}{Theorem}[section]

\newtheorem{prop}[theorem]{Proposition}
\newtheorem{definition}[theorem]{Definition}

\theoremstyle{remark}
\newtheorem{remark}[theorem]{Remark}
\newtheorem{example}[theorem]{Example}
\renewcommand{\a}{\alpha}
\renewcommand{\b}{\beta}
\newcommand{\g}{\gamma}
\newcommand{\G}{\Gamma}
\renewcommand{\d}{\delta}
\newcommand{\D}{\Delta}

\newcommand{\z}{\zeta}

\renewcommand{\k}{\kappa}

\newcommand{\m}{\mu}

\newcommand{\s}{\sigma}

\renewcommand{\t}{\tau}
\newcommand{\f}{\phi}

\renewcommand{\P}{\Pi}

\renewcommand{\O}{\Omega}

\newcommand{\C}{{\mathbb C}}

\newcommand{\Z}{{\mathbb Z}}
\newcommand{\N}{{\mathbb{N}}}

\newcommand{\DD}{{\mathbb{D}}}

\newcommand{\LL}{\mathbb{L}}

\newcommand{\cb}{{\mathbf c}}

\newcommand{\mb}{{\mathbf m}}
\newcommand{\nb}{{\mathbf n}}

\newcommand{\Cb}{{\mathbf C}}

\newcommand{\Fb}{{\mathbf F}}

\newcommand{\Ib}{{\mathbf I}}

\newcommand{\Kb}{{\mathbf K}}

\newcommand{\Pb}{{\mathbf P}}

\newcommand{\Sbb}{{\mathbf S}}
\newcommand{\Tb}{{\mathbf T}}
\newcommand{\Ub}{{\mathbf U}}

\newcommand{\AF}{\mathfrak A}

\newcommand{\MF}{\mathfrak M}

\newcommand{\Ac}{{\mathcal A}}
\newcommand{\Bc}{{\mathcal B}}

\newcommand{\Dc}{{\mathcal D}}
\newcommand{\Ec}{{\mathcal E}}

\newcommand{\Lc}{{\mathcal L}}
\newcommand{\Mcc}{{\mathcal M}}

\newcommand{\dist}{{\rm dist}\,}

\newcommand{\supp}{\hbox{{\rm supp}}\,}

\DeclareMathOperator{\im}{{\rm Im}\,}

\newcommand{\Diag}{\operatorname{Diag\,}}

\newcommand{\ccap}{\operatorname{cap}}

\theoremstyle{plain}

\newtheorem*{theorem*}{Theorem}

\theoremstyle{definition}

\newcommand{\la}{\langle}
\newcommand{\ra}{\rangle}

\newcommand{\beq}{\begin{equation}}
\newcommand{\eeq}{\end{equation}}

\numberwithin{equation}{section}
\numberwithin{figure}{section}

\begin{document}

\title[Polyanalytic Bergman spaces]{Toeplitz operators with singular symbols in polyanalytic Bergman spaces on the half-plane}

\author{Grigori Rozenblum }

\address{ Chalmers University of Technology and The University of Gothenburg (Sweden); St.Petersburg State University, Dept. Math. Physics (St.Petersburg, Russia)}

\email{grigori@chalmers.se}
\author {Nikolai Vasilevski}
\address{Cinvestav (Mexico, Mexico-city)}
\email{nvasilev@cinvestav.mx}

\subjclass[2010]{47A75 (primary), 58J50 (secondary)}
\keywords{Bergman space, Toeplitz operators}
\dedicatory{From the first named author, with best wishes to the second named one, \\
on the occasion of his Jubilee}
\begin{abstract}
Using the approach based on sesquilinear forms, we introduce Toeplitz operator in the analytic Bergman  space on the upper half-plane with strongly singular symbols, derivatives of measures. Conditions for boundedness and compactness of such operators are found. A procedure of reduction of  Toeplitz operators in Bergman spaces of polyanalytic functions to  operators with singular symbols in the analytic Bergman space by means of the creation-annihilation structure is elaborated, which leads to the description of the properties of the former operators.
\end{abstract}
\thanks{The first-named author is supported by grant  RFBR No 17-01-00668. }
\maketitle



\section{Introduction}
In a series of papers \cite{RV1,RV2,RV3} the authors developed the approach to defining the Toeplitz operators in Bergman type spaces by means of bounded sesquilinear forms, which permitted them to study Toeplitz operators with strongly singular symbols. The cases of the classical Bergman space of analytic functions on the unit disk  $\mathbb{D}$ and the Fock space on the whole complex plane were considered. The present paper is devoted to the study of such Toeplitz operators in one more classical Bergman space, the one of analytic functions on the upper half-plane  $\Pi$, square integrable with respect to the Lebesgue measure.  It is well known that for sufficiently regular, say, bounded symbols, the theories of Toeplitz operators in the Bergman spaces on the disk and on the upper half-plane are equivalent, in the sense that the M\"obius transform
\begin{equation}\label{Mobius}
    z\mapsto \z =M(z):=\frac{z-i}{1-iz}: \Pi\to\DD,
\end{equation}
generates the mapping
\begin{equation}\label{Mobius isometry}
    \Ub: f\mapsto g=\Ub f, (\Ub f)(z)=f(M(z))\frac{1}{(1-iz)^2}, z\in \Pi,
\end{equation}
which is an isometry of the Bergman spaces $\Ac^2(\DD)$ and $\Ac_1:=\Ac^2(\Pi)$. However, when passing to more singular symbols that generate Toeplitz operators via sesquilinear forms, the boundedness conditions carry over not in such simple way.

In the present paper, we introduce Carleson measures for derivatives of order $k$ ($k$-C measures) for the Bergman space on the half-plane and find  conditions for a given  measure to be a $k$-C measure. As usual, these conditions are sufficient for complex measures but are also necessary for positive ones. An estimate for the properly defined norm of $k$-C measures, with explicit dependence on the derivative order $k$, is obtained. This makes it possible to consider strongly singular symbols, containing derivatives of unbounded order.

The above results are applied for studying Toeplitz operators in polyanalytic Bergman spaces on the upper half-plane, using the creation-annihilation structure discovered in \cite{KarlPessIEOT,V3}.

The first-named author is grateful to the Mittag-Leffler institute for hospitality and support while a considerable part of the paper was written.

\section{The Bergman space on the half-plane and related structures. Singular symbols}
Similar to the hyperbolic metric on the disk, the pseudo-hyperbolic metric on the half-plane is useful.
As known, for $z,w\in\Pi$, the pseudo-hyperbolic distance between $z,w$ is defined by
\begin{equation*} 
    d(z,w)=\left|\frac{z-w}{z-\overline{w}}\right|.
\end{equation*}
We will denote by $D(z,R)$, $z=x+iy, R<1$, the pseudo-hyperbolic disk in $\Pi$, centered at $z$ with 'radius' $R$. It is easy to check that the disk $D(z,R)$ coincides with the Euclidean disk $B(w,r)$,
\begin{equation} \label{Disks}
    D(x+iy,R)=B(w,r), \ \ \textrm{where} \ \ w=x+i\frac{1+R^2}{1-R^2}y, \ \ r=\frac{2R y}{1-R^2}.
\end{equation}
Thus, the area of the disk $D(z,R)$ equals $\frac{4\pi R^2 y^2}{(1-R^2)^2}$.

Conversely, the Euclidean disk $B(x+i\eta, r) \in \Pi$ is the pseudo-hyperbolic disk $D(x+iy,R)$ with
\begin{equation*}
 R=\frac{\eta}{r}-\sqrt{\frac{\eta^2}{r^2}-1}\quad \textrm{and} \quad y =\frac{1-R^2}{1+R^2}\eta.
\end{equation*}
Note that, while the pseudo-hyperbolic radius is fixed, the $y$-coordinate of the center of the pseudo-hyperbolic disk is proportional to the corresponding coordinate of the Euclidean disk.

Recall that the Bergman space $\Ac_1=\Ac^2(\Pi)$ is a subspace in $L^2(\Pi)$ which consists of analytic in $\Pi$ functions. It is a reproducing kernel space, with the reproducing kernel $\k(z,w)=-(\pi (z-\overline{w})^2)^{-1}.$ Thus, the integral operator $\Pb=\Pb_{\Pi}$ with kernel $\k(z,w)$ is the orthogonal (Bergman) projection of $L^2(\Pi)$ onto $\Ac^2(\Pi)$. This projection is connected with the Bergman projection $\Pb_{\DD}$ for the Bergman space $\Ac^2(\DD)$ by means of the operator $\Ub$ in \eqref{Mobius isometry}:
\begin{equation}\label{equivProj}
    \Pb_{\Pi}=\Ub^* \Pb_{\DD}\Ub.
\end{equation}

Given a bounded function $a(z), \, z\in \Pi$, the Toeplitz operator in $\Ac_1$ is defined in the usual way as
\begin{equation}\label{ToeplB}
    (\Tb_{a}f) (z)=(\Pb a f)(z)= \int_{\Pi} \k(z,w) a(w)f(w) dA(w),
\end{equation}
$dA$ being the Lebesgue measure. This operator is bounded in $\Ac_1$, as a composition of two bounded operators. The function $a(z)$ is called the \emph{symbol} of the Toeplitz operator. It was the  object of many studies to extend  this definition to symbols being  objects, more singular than bounded functions, still producing bounded operators. This program was implemented, in particular, in \cite{RV1,RV2} for the Bergman space on the disk and for the Fock space. Now we follow the pattern of these papers for the upper half-pane case.

The first stage here is considering (complex) measures as symbols. Let $\m$ be an absolutely continuous measure on $\Pi$, with a bounded density $a(z)$ with respect to the Lebesgue measure $dA$. Then the action of operator \eqref{ToeplB} can be written as
\begin{equation}\label{ToeplMB}
    (\Tb_{a}f) (z)=(\Pb a f)(z)= \int_{\Pi} \k(z,w) a(w)f(w) dA(w)=\int_\Pi \k(z,w) f(w)d\m(w).
\end{equation}
Such a definition of the Toeplitz operator by  means of the expression on the right-hand side of \eqref{ToeplMB}  can be, at least formally, extended to measures $\mu$ which are not necessarily absolutely continuous with respect to $A$ with bounded density, and even to those that are not absolutely continuous with respect to $A$ at all; what is, actually, needed is just the boundedness of the operator defined by the right-hand side in \eqref{ToeplMB}. To find some effective analytical conditions for this boundedness is rather a quite hard task. At the same time the approach based upon sesquilinear forms turns out to be very efficient here. In fact, in case of $d\mu=a(z)dA(z)$, we consider the sesquilinear form $\Fb_{\m}[f,g]=\la \Tb_{a}f,g\ra$, where $f,g\in \Ac_1$. By the above definition of the operator $\Tb_{a}$,
\begin{gather}\label{Form1}
 \Fb_{\m}[f,g]=\la \Pb a f,g\ra=\la af,\Pb g\ra=\la a f,g\ra =\\ \nonumber\int_{\Pi}a(z)f(z)\overline{g(z)}dA(z)=\int_{\Pi}f(z)\overline{g(z)}d\mu(z), \ \ f,g\in\Ac_1.
\end{gather}
The left-hand side in \eqref{Form1} is defined for $a$ being a (sufficiently nice) function, however, the right-hand side makes sense for a measure $\m$ and can be thus used for a definition of a Toeplitz operator. The sesquilinear form \eqref{Form1} defines a bounded operator in $\Ac_1$ in case it is bounded, i.e., $|\Fb_{\m}[f,g]|\le C\|f\|\|g\|$. This boundedness follows, as soon as the  inequality
\begin{equation}\label{Car}
    \left|\int_{\Pi}|f(z)|^2d\m(z)\right|\le C\|f\|^2
\end{equation}
is satisfied for all $f\in\Bc$. This estimate is, surely, satisfied provided
\begin{equation*} 
    \int_{\Pi}|f(z)|^2d|\m(z)|\le C\|f\|^2,
\end{equation*}
where $|\m|$ denotes the variation of the measure $\m$. Note that the considerations involving sesquilinear forms are essentially more convenient in analysis since they  evade using reproducing kernels and deal  with inequalities containing only the functions $f,g$ and the measure $\m.$

Measures $\m$ subject to the estimate \eqref{Car} are called Carleson measures for the space $\Ac_1$. A description for such Carleson measures was given for the case of the Bergman  space on the disk, for the Fock space and some other Bergman and Hardy type spaces, see, e.g., \cite{Zhu,Zhu2,ZhuFock}.  The criterion for a measure to be a Carleson measure for $\Ac_1$ follows from the known one for the space $\Ac(\DD)$.

\begin{prop}\label{PropCarlMeasure}
Let $\m$ be a complex  Borel measure on $\Pi$. For a fixed $R\in(0,1)$, if
\begin{equation}\label{CarlSuff}
|\m|(D(z,R))\le C_\m A(D(z,R)),\quad z\in \Pi, \ \ R<1,
\end{equation}
with constant $C_\m$  not depending on $z$, then the sesquilinear form \eqref{Form1} is bounded in $\Ac_1$. That is, the inequality \eqref{Car} is satisfied for all $f\in \Ac_1$, with constant $C=C(R)C_\m$, where $C(R)$ depends only on $R$. If the measure $\m$ is positive, the condition \eqref{CarlSuff} is also necessary for the the sesquilinear form \eqref{Form1} to be bounded.
\end{prop}
\begin{proof}The result follows from a similar statement concerning Carleson measures for the Bergman space on the disk, see, e.g., \cite{Zhu}, by means of the unitary equivalence \eqref{equivProj}. A reasoning establishing directly this property  can be found, for example, in \cite[Proposition 2.6]{Kang}.
\end{proof}

We note here that the condition \eqref{CarlSuff} can be (although just formally !) relaxed, when replaced by
\begin{equation}\label{CarlSuff1}
|\m|(D(z,R))\le C_\m A(D(z,R_1)),\quad z\in \Pi, \ \ R<1,
\end{equation}
with any fixed $R_1>R$. This remark will be used when considering Carleson measures for derivatives later on.
A measure $\m$ on $\Pi$, having compact support, can be considered as a distribution in $\Ec'(\Pi)$. At the same time, the function $f(z)\overline{g(z)}$ is infinitely differentiable in $\Pi$, $f(z)\overline{g(z)}\in \Ec(\Pi)$, moreover, it is real-analytic in $\Pi$. Thus, this function can be represented as
\begin{equation*}
    f(z)\overline{g(z)}=\Diag^*(f\otimes \bar{g})=\Diag^*((f\otimes 1)(1\otimes \bar{g})),
\end{equation*}
where $\Diag$ is the diagonal embedding of $\Pi$ into $\Pi\times\Pi$, $\Diag(z)=(z,z)$, and $\Diag^{*}$ is the induced mapping
$\Ac_1\otimes\overline{\Ac_1}$ to $\AF(\Pi)$ (the space of real-analytic functions), $\Diag^*(f\otimes\overline{g})=f(z)\overline{g(z)}$. We denote by $\Mcc$ the image of $\Ac_1\otimes\overline{\Ac_1}$ in $\AF(\Pi)$ under the mapping $\Diag^*$.
 Therefore the expression \eqref{Form1} can be understood as
\begin{equation}\label{FormDistr}
   \Fb_{\m}[f,g]=(\m, f(z)\overline{g(z)})=(\m,h), \ \, \text{with} \ \, h=\Diag^*(f\otimes\bar{g})\in\AF(\Pi),
\end{equation}
where parentheses denote the intrinsic paring of $\Ec'(\Pi)$ and $\Ec(\Pi)$. Proposition~\ref{PropCarlMeasure} can be now understood in the sense that as soon as the condition \eqref{CarlSuff} is satisfied, the sesquilinear form \eqref{FormDistr} can be extended to $\Mcc$ for the measure $\m$ not necessarily having compact support.
Moreover, the estimate $|(\m,h)|\le C_\m\|f\|\|g\|$, $h\in\Mcc$, holds. We recall here how such extension of the distribution is standardly performed; this will make  our further considerations for distributions of more general nature more clear.

Let $\m_j$ be a sequence of measures with compact support, each one in a (closed) quasi-hyperbolic disk $D_j\varsubsetneq\Pi$ of radius $r$, and let these disks form a covering of $\Pi$ with finite multiplicity $m$: each point of $\Pi$ is covered by not more than $m$ disks $D_j$. Suppose that for each measure $\m_j$, the estimate $|\Fb_{\m_j}[f,g]|\le C\|f\|\|g\|$ is satisfied, with the same constant $C$. Then, we can sum up such estimates over $j$ and, using the finite multiplicity property, arrive at the same estimate for the measure $\m=\sum\m_j$, automatically locally finite in $\Pi$ (just with a controllably larger constant.) This line of reasoning, considering first the distributions in $\Ec'(\Pi)$, with compact support, obtaining proper estimates, and extending then these estimates to certain distributions without compact support condition, so that the estimates hold on $\Ac_1$, will be further implemented for more general distributions.

 A few words to explain our philosophy. Let $\O$ be an open subset in $\C$ ($\O=\Pi$ in our case). If $F$ is a distribution in $\Ec'(\O)$ i.e., a distribution with compact support in $\O$, its derivative, say, $\partial F$ is standardly defined as the distribution $\partial F$ acting on functions $\f\in \Ec(\O)$ by the rule $(\partial F, \f)=-(F, \partial \f)$. If, however, $F$ is a distribution in a wider space, $F\in  \Dc'(\O)$, the action $(\partial F, \f)$ is not necessarily defined for all $\f\in\Ec(\O)$. In particular, if $\f$ is a nontrivial function in the Bergman space or a function in $\Mcc$, it can never belong to $\Dc(\O)$, so the action of $F$ on such functions and, further on, the definition of derivatives of $F$ needs to be specified anew, however being consistent with the usual definition. In what follows, we consider a class of distributions for which such construction works, preserving the usual properties of distributions. The natural compensation for this frivolity is the narrowing of the
set of functions on which such 'distributions' act.

\section{Carleson measures for derivatives}
Following the pattern in \cite{RV1,RV2}, we introduce now a class of sesquilinear forms involving derivatives of functions $f,g$, corresponding thus to (formal) distributional derivatives of Carleson measures.

\begin{definition} Let $\m$ be a regular fix sign measure on $\Pi$ and $\a,\b$ be two nonnegative integers. We denote by $\Fb_{\a,\b,\m}$ the sesquilinear form
\begin{equation}\label{FormDer}
   \Fb_{\a,\b,\m}[f,g]=(-1)^{\a+\b}\int_{\Pi}\partial^{\a}f(z)\overline{\partial^\b g(z)}d\m(z), \ \ f,g\in\Ac_1,
\end{equation}
which we denote as well by
\begin{equation}\label{FormDer1}
    \Fb_{\a,\b,\m}[f,g]=(\partial^{\a}\bar{\partial}^{\b}\m, f\bar{g}).
\end{equation}

\end{definition}
This definition is consistent with our explained above approach. In fact \eqref{FormDer}, \eqref{FormDer1} act as the \emph{definition} of the action of the 'distribution' $\partial^{\a}\bar{\partial}^{\b}\m$ on elements in $\Ac_1$. Note that this is consistent with the standard distributional definition of the derivative of a measure in the case when $\m$ is a compactly supported measure in $\Pi.$

The first set of properties  of such forms and corresponding operators, similar to the ones for other Bergman spaces, is the following.

\begin{theorem} Let $\m$ be a measure with compact support in $\Pi$ and let  $\a,\b$ be nonnegative integers.
Then \begin{enumerate}
\item For any $f,g\in\Ac_1$, the integral in \eqref{FormDer} converges, moreover,
\begin{equation*}
   \Fb[f,g]= \Fb_{\a,\b,\m}[f,g]=(\partial^{\a}\bar{\partial}^{\b}\m, f\bar{g}),
\end{equation*}

where the derivatives are understood in the sense of distributions in $\Ec'(\Pi)$ and the parentheses mean the intrinsic paring in $(\Ec'(\Pi), \Ec(\Pi))$.
\item The sesquilinear form \eqref{FormDer} is bounded, considered on $\Ac_1\times \Ac_1$ and therefore defines a bounded Toeplitz operator by $\la \Tb_\Fb f,g\ra=\Fb[f,g]$, or \\ $(\Tb_{\Fb}f)(z)=\Fb[f,\k_z(\cdot)]$, where, as usual, $\k_z(w)=\overline{k(z,w)} = k(w,z)$.
\item The sesquilinear form \eqref{FormDer} determines a compact Toeplitz operator  $\Tb_{\Fb}$ in $\Ac_1 $.
\item If $s_n(\Tb_{\Fb})$ denote the singular numbers of the operator $\Tb_{\Fb}$, then the following estimate holds
\begin{equation}\label{snumbers}
    s_n(\Tb)\le C \exp(-n\s), \ \ \ n \in \mathbb{N},
\end{equation}
where $\s>0$ is a constant determined by the measure $\m$ and integers $\a,\b$.
\end{enumerate}
\end{theorem}
\begin{proof} The property (4) absorbs the other ones, so we will prove only it. Due to Ky Fan's inequalities for singular numbers of compact operators, it is sufficient to establish \eqref{snumbers} for a positive measure $\m$.
Consider a closed Euclidean disk $B\subset \Pi$ with radius $R$ such that for some $r>0$, the support of $\m$ lies strictly inside $B$, thus $\dist(z,\partial B)>r>0$ for all $z\in\supp \m$.
The Cauchy integral formula implies that for any $z\in\supp \m$ and any $\a\in\Z_+$,
\begin{equation*}
|\partial^\a f(z)|^2\le C_{\a}\int_{\partial B}|f(\z)|^2dl(\z)
\end{equation*}
for each function $f\in\Ac_1$, with constant $C_\a$ depending only on  $\a$, but not depending on $f$ and $z$. By the same reason,  for any $z\in\supp\m$, the estimate
\begin{equation}\label{est in disk2}
 |\partial^\a f(z)|^2\le C_{\a}\int_{|\z-z|\in(R,R+r/2)}|f(\z)|^2dA(\z)
\end{equation}
holds.  By the Cauchy-Schwartz inequality,
\begin{equation*}
    |\Fb_{\a,\b,\m}[f,g]|\le \left(\int_{\supp \m}|\partial^\a f(z)|^2 d|\m|(z)\right)^{\frac12}\left(\int_{\supp \m}|\partial^\b g(z)|^2 d|\m|(z)\right)^{\frac12},
\end{equation*}

for all $f,g\in\Ac_1(\Pi)$, and then, due to \eqref{est in disk2},
\begin{equation*}
  |\Fb_{\a,\b,\m}[f,g]|\le C'_{\a}C'_{\b}|\m|(B)\|f\|_{L^2(B')}\|g\|_{L^2(B')}.
\end{equation*}

The last relation means that the sesquilinear form $\Fb_{\a,\b,\m}$ is bounded not only in $\Ac_1(\Pi)$, but in $\Ac_1(B')$ as well, where $B'$ is the Euclidean  disk $B'=B(z,R+r/2)$.
Now we represent the Toeplitz operator $\Tb_\Fb$ as the composition
\begin{equation}\label{est5}
    \Tb_{\Fb}=\Tb_{\Fb_{\Ac_1(B')}}\Ib_{B\Rrightarrow B'}\Ib_{B'\Rrightarrow\Pi},
\end{equation}
where $\Ib_{B\Rrightarrow B'}:\Ac_1(B')\to\Ac_1(B)$, $\Ib_{B'\Rrightarrow \Pi}:\Ac_1(\Pi)\to\Ac_1(B')$ are  operators generated by restrictions of functions defined on a larger set to the corresponding smaller set. The equality \eqref{est5} can be easily checked by writing the sesquilinear forms of operators on the left-hand and on the right-hand side. Finally, the first and the third operators on the right-hand side are bounded, while the middle one, the operator generated by the embedding of the disk $B$ to $B'$, is known to have the exponentially fast decaying sequence of singular numbers see, e.g., \cite{parfenov}.
\end{proof}

\begin{remark} Using the results in \cite{parfenov}, one can give an upper estimate for the constant $\s$ in \eqref{snumbers}. In fact, let $Q\subset\DD$ be the image of $\supp \m$ in the unit disk, under the mapping \eqref{Mobius}. We denote by $\Cb(Q)$ the  set of \emph{connected} closed sets $V\subset \DD$ containing $Q$. For each $V\in\Cb(Q)$, let $\ccap(V)$ denote the logarithmic capacity of $V$ (see the definition, e.g., in \cite{parfenov}) and we set $\cb(\m)$ as
\begin{equation*}
    \cb(\m)=\inf_{V\in \Cb(Q)}\ccap(V).
\end{equation*}

Then  estimate \eqref{snumbers} holds for any $\s<\cb(\m)$.
In fact, for a \emph{positive} measure $\m$, $\a,\,\b=0$, and for a connected set $Q$, the asymptotics of the singular numbers of the Toeplitz operator was found in \cite{parfenov}, and our estimate for the exponent in \eqref{snumbers} follows from this result and the natural monotonicity of singular numbers under the extension of the set where the measure is supported. Our considerations give only the upper estimate for these singular numbers. It is remarkable that for a non-connected support of the measure, even in the setting of \cite{parfenov}, the terms in which the singular numbers asymptotics or even sharp order estimates can be expressed are so far unknown.
\end{remark}
Next, we present our main definition.
\begin{definition}\label{DefK-C} A measure on $\Pi$ is called a Carleson measure for derivatives of order $k$ ($k$-C measure) if for some constant $C_k(\m)>0$,
\begin{equation}\label{kForm}
    |\Fb_{k,\m}[f,f]|\equiv\left|\int_{\Pi}|\partial^k f(z)|^2 d\m(z)\right|\le C_k(\m)\|f\|^2
\end{equation}
for all $f\in \Ac_1$.
\end{definition}

Now, we find a sufficient condition for a measure to be a $k$-C measure. Here it is important to control the dependence of the constant $C_k(\m)$ in \eqref{kForm} on the number $k$.
\begin{theorem}\label{main theorem}
Let $\g\in(0,1)$ be  fixed, and let the measure $\mu$ on $\Pi$ satisfy the condition
\begin{equation}\label{Cond kC}
    |\m|(D(z_0,\g))\le \varpi_\k(\m) |\im z_0|^{2k}A(D(z_0,\g))
\end{equation}
with some $\varpi_\k(\m)$ for all $z_0\in\Pi$.
Then inequality \eqref{kForm} is satisfied, with $C_k(\m)=(k!)^2\varpi_\k(\m)\g^{-2k}$, that is, $\m$ is a $k$-C measure.
\end{theorem}
\begin{proof} Given a point $z_0\in\Pi$, the pseudo-hyperbolic disk $D(z_0,\g)$ coincides with the Euclidean  disk $B(w_0,s)$, where their centra and radii are connected by \eqref{Disks}. We write then the standard representation of the derivative at a point $w$, $|w_0-w|<s_1<s$ of an analytic function $f(w)$:
\begin{equation}\label{representation}
f^{(k)}(w)=k!(2\pi\imath)^{-1}\int_{|w_0-\z|=\s}(\z-w)^{-k-1}f(\z)d\z, \ \ s_1<\s<s.
\end{equation}
Now we fix $s_2\in(s_1,s)$ and integrate \eqref{representation}  in $\s$ variable from $s_2$ to $s$, which gives us the estimate
\begin{equation*}
|f^{(k)}(w)|\le (s-s_2)^{-1}k!(2\pi)^{-1}\int_{s_2\le |w_0-\z|\le s }|\z-w|^{-k-1}|f(\z)|dA(\z).
\end{equation*}
We choose then $s_1=\frac14 s$, $s_2=\frac34 s$.
The inequality $|\z-w| \geq \frac12 s$,  together with the Cauchy-Schwartz, one yields
\begin{gather*} |f^{(k)}(w)|\le (s/2)^{-k-2}k!\pi^{-1}\int_{|w_0-\z|\le s} |f(\z)|dA(\z) \nonumber\\ \le \frac{2}{\sqrt{\pi}} (s/2)^{-(k+1)}k!\left(\int_{|w_0-\z|\le s} |f(\z)|^2dA(\z)\right)^{\frac12},
\end{gather*}
or
\begin{equation}\label{repr2}
    |f^{(k)}(w)|^2\le (4/\pi) (s/2)^{-2(k+1)}(k!)^2\,\int_{|w_0-\z|\le s} |f(\z)|^2dA(\z).
\end{equation}

The estimate \eqref{repr2} holds for all $w$, $|w-w_0|<s_1=s/4$. Therefore we can integrate it over the Euclidean disk $B(w_0,s_1)$ with respect to the measure $\m$, which gives
\begin{eqnarray}\label{repr3} \nonumber
 &&\left| \int_{B(w_0,s_1)}|f^{(k)}(w)|^2 d\m(w)\right|\\
&\le& |\m|(B(w_0,s_1)) (4/\pi) (s/2)^{-2(k+1)}(k!)^2\,\int_{|w_0-\z|\le s} |f(\z)|^2dA(\z).
\end{eqnarray}
 We substitute now the estimate \eqref{Cond kC} into \eqref{repr3} and use \eqref{Disks} to arrive at
 \begin{eqnarray*}
  &&\left| \int_{B(w_0,s_1)}|f^{(k)}(w)|^2 d\m(w)\right| \\
  && \leq \varpi_k(\m)\, y_0^{2k}\pi (s/2)^2 (4/\pi)(s/2)^{-2(k+1)}(k!)^2\,\int_{|w_0-\z|\le s} |f(\z)|^2dA(\z) \\
  && \leq  \varpi_k(\m)\,  \g^{-2k}(k!)^{2}\int_{|w_0-\z|\le s} |f(\z)|^2dA(\z),
 \end{eqnarray*}
or, returning to the pseudo-hyperbolic disks,
\begin{equation}\label{repr4}
  \left | \int_{D(z_1, \g_1)}|f^{(k)}(w)|^2 d\m(w)\right|\le  \varpi_k(\m)\,  \g^{-2k}(k!)^{2}\int_{D(z_0,\g)} |f(\z)|^2dA(\z),
\end{equation}
where $D(z_1, \g_1) = B(w_0,s_1) \subset B(w_0,s) = D(z_0,\g)$.

Now we follow the reasoning in \cite[Theorem 7.4]{Zhu}, where estimates of the type \eqref{repr4} were summed to obtain the required $k$-Carleson property. One just should replace the hyperbolic disks with pseudo-hyperbolic ones.
Like in \cite{Zhu}, it is possible to find a locally finite covering $\Xi$ of the half-plane $\Pi$ by disks of the type $D(z_1,\g_1)$, with $z_1\in \Pi$, so that the larger  disks $D(z_0,\g)$ form a covering $\widetilde{\Xi}$ of $\Pi$ of finite, moreover, controlled  multiplicity. The latter  means that the number
 $m(\widetilde{\Xi})=\max_{z\in\Pi}\#\{D\in \widetilde{\Xi}:z\in D\}$
is finite.  After adding up all inequalities of the form \eqref{repr4} over all disks $D=D(z_1,\g_1)\in \Xi$, we obtain the required estimate for $\int_{\Pi}|f^{(k)}|^2d\mu$. \end{proof}
Having Theorem \ref{main theorem} at our disposal, we introduce the classes of $k$-C measures.
\begin{definition}\label{Def Mk}Fix a number $\g\in(0,1)$. The class $\MF_{k,\g}$ consists of measures $\m$ on $\Pi$ such that
\begin{equation*}
   \varpi_{k,\g}(\m):= \sup_{z\in\Pi}\{|\m|(D(z,\g ))(\im z)^{-2(k+1)}(k!)^2\g^{-2k}\}<\infty.
\end{equation*}

\end{definition}
Theorem \ref{main theorem} implies that the class $\MF_{k,\g}$ consists of $k$-C measures. This class, in fact, does not depend on the value of $\g$ chosen, however the value $\varpi_{k,\g}(\m)$ does.

It is convenient to extend the definition of $\MF_{k,\g}$ to half-integer values of $k$:

\begin{definition}\label{DefKCMM}Let $k\in\Z_++\frac12$ be a half-integer. The class $\MF_{k,\g}$, $\g\in(0,1)$ consists of measures satisfying
\begin{equation}\label{Mk12}
   \varpi_{k,\g}(\mu) :=  \sup_{z\in{\Pi}}\left\{ |\m|\left(D(z, \g)\right)(\im{z})^{-2(k+1)}\G(k+1)^2 \g^{-2k}\right\}<\infty.
\end{equation}
\end{definition}
Further on, the parameter $\g$ will be fixed and will be often omitted in our notations. The quantity $\varpi_{k}(\mu)$ will be called the $k$-norm of the measure $\m$.
According to  Definition \ref{DefKCMM}, the spaces $\MF_{k,\g}$ for different $k\in \Z_+/2$ are related by
\begin{equation*}
    \MF_{k}=(\im{z})^{2(l-k)}\MF_{l},
\end{equation*}
with
\begin{equation*}
    \varpi_k(\m)=\g^{2(l-k)}(\G(k+1)/\G(l+1))^2\varpi_l(\m).
\end{equation*}

Theorem \ref{main theorem} enables us to find   conditions for boundedness of the operators defined by differential sesquilinear forms $\Fb_{\a,\b,\m}$ in \eqref{FormDer}. They look similar to the corresponding conditions for sesquilinear forms in the Bergman space on the unit disk and are derived from Theorem \ref{main theorem} in the same way as Proposition 6.8 is derived from Theorem 6.3 in \cite{RV2}, so we restrict ourselves to formulations only.

\begin{theorem}\label{Th.concrete}
Let the measure $\m$ satisfy \eqref{Mk12} with some $k\in\Z_+/2$. Then for $\a,\b\in \Z_+, \,\a+\b=2k$, the sesquilinear form
\begin{equation*}
    \Fb_{\a,\b,\m}[f,g]=(-1)^{\a+\b}\int_{\Pi}\partial^\a f\overline{\partial^\b g}d\m
\end{equation*}
is bounded in $\Ac_1$ and defines a bounded Toeplitz operator $\Tb_{\a,\b,\b}$ in the Bergman space $\Ac_1$, moreover its norm is majorated by $\varpi_{k}(\mu)$.
\end{theorem}
Taking into account our above agreement concerning distributional derivatives of measures without compact support condition, Theorem \ref{Th.concrete} can be reformulated as
\begin{theorem}\label{Th.concrete.der} Under the conditions of Theorem  \ref{Th.concrete}, the sesquilinear form
\begin{equation*}
  \Fb_{\partial^\a\overline{\partial^\b}\m}[f,g]=(\partial^\a\overline{\partial^\b}\m, f\bar{g})
\end{equation*}
is bounded in $\Ac_1$ and defines a bounded Toeplitz operator $\Tb_{\partial^\a\overline{\partial^\b}\m}$ in $\Ac_1$.
\end{theorem}

As usual for Toeplitz type operators, boundedness conditions lead to compactness conditions, formulated in similar terms.
\begin{theorem}\label{Th.Comp} For $R>0$, denote by $Q_R$ the rectangle in $\Pi$:
\begin{equation*}
    Q_R=\{z=x+y\in\Pi: x\in(-R,R), y\in(R^{-1}, R)\}.
\end{equation*}
Suppose that $\a+\b=2k$ and
\begin{equation}\label{Mk12c}
   \lim_{R\to \infty}  \sup_{z\in{\Pi\setminus Q_R}}\left\{ |\m|\left(D(z, \g)\right)(\im{z})^{-2(k+1)}\G(k+1)^2 \g^{-2k}\right\}=0.
\end{equation}
Then the operator $\Tb_{\partial^\a\overline{\partial^\b}\m}$ is compact in $\Ac_1$.
\end{theorem}

\begin{proof}
 It goes in a standard way. Split the measure $\m$ into two parts, $\m=\m_R+\m_R'$, where $\m_R$ has support outside $Q_R$ and $\m_R'$ has compact support.  Correspondingly, the operator $\Tb_{\partial^\a\overline{\partial^\b}\m}$ splits into two terms, $\Tb_R+\Tb_R'$. The first operator, by Theorem \ref{Th.concrete.der}, has small norm, as soon as $R$ is chosen sufficiently large. The operator $\Tb_R'$ is compact by Theorem 3.1. Therefore, $\Tb_{\partial^\a\overline{\partial^\b}\m}$ is compact.
\end{proof}

\section{Examples}We give some examples of symbols--distributions and --hyperfunctions. More examples can be constructed, following the pattern seen in \cite{RV1}, \cite{RV2}.
\begin{example}\label{Ex1}Let the measure $\m$ be supported on the lattice $\LL=\Z+ i \N$. With an integer point $\nb=(n_1+in_2)$ we assign the weight $\mb_{\nb}>0$. Suppose that $\sup_{\nb}\mb_{\nb}<\infty$. Then the Toeplitz operator, with the measure $\m=\sum_{\nb}\mb_{\nb}\d(z-\nb)$ as symbol, is bounded.
\end{example}
\begin{example}\label{Ex2}In the setting of Example 4.1, consider the Toeplitz operator with distributional symbol $\m_{\a,\b,W}=W(\nb)\partial^\a\overline{\partial^\b}\m$, where $W(\nb)$ is a weight function, $W(n_1+in_2)=|n_2|^{-\a-\b}$. Then the conditions of Theorem \ref{Th.concrete.der} are satisfied for the measure $W(\nb)\m$ and, therefore, the Toeplitz operator with symbol $\m_{\a,\b,W}$ is bounded. If $W(\nb)$ is a function on the lattice satisfying $W(\nb)\to 0$ as $|\nb|\to \infty$ then, by Theorem \ref{Th.Comp}, this Toeplitz operator is compact.
\end{example}\label{Ex3}
\begin{example}\label{ex4.3}
In the setting of Example 4.1, consider the, initially formal, sum
\begin{equation*}
    a=\sum_{\nb\in\LL} W(\nb)\partial^{\a_\nb}\overline{\partial^{\b_{\nb}}}\m(\{\nb\}),
\end{equation*}
where $(\a_{\nb}, \b_{\nb})$ is a collection of orders of differentiation and $W(\nb)$ is a weight sequence. This  is a sum of distributions supported at single points of the lattice $\LL$. To each of them, we can apply Theorem \ref{Th.concrete.der} and obtain an estimate of the norm of the corresponding sesquilinear form $\Fb_{\nb}$, where the order $k_{\nb}=\a_{\nb}+\b_{\nb}$ is involved:
\begin{eqnarray}\label{ex3.est} \nonumber
    |\Fb_{\nb}[f,g]|&\le& C |W(\nb)|((\a_{\nb}!\b_{\nb}!)(\g/2)^{-(\a_{\nb}+\b_{\nb})}n_2^{-\a_{\nb}+\b_{\nb}}\|f\|\|g\| \\
    &=&C |W(\nb)|\t(\nb)\|f\|\|g\|.
\end{eqnarray}
A rough way to estimate the sesquilinear form $\Fb_a$ would be to consider the sum of the terms in \eqref{ex3.est},
\begin{equation}\label{ex3.est2}
|\Fb_a[f,g]|\le \sum(|W(\nb)|\t(\nb))\|f\|\|g\|,
\end{equation}
so the sesquilinear form is bounded as soon as the series in \eqref{ex3.est2} converges. A more exact treatment uses the finite multiplicity covering by disks containing no more than, say, 10 points of the lattice $\LL$ similarly to how this was done in Section 2. In this way, the sesquilinear form is majorated by a smaller quantity,
\begin{equation*}
    |\Fb_a[f,g]|\le \sup\{(|W(\nb)|\t(\nb))\|f\|\|g\|\},
\end{equation*}
for which the finiteness condition requires a considerably milder decay requirement for $W(\nb)$ than the finiteness of the coefficient in \eqref{ex3.est2}.
\end{example} Now we consider some symbols with support touching the boundary of the upper half-plane $\Pi$.
\begin{example}\label{Ex4}  Let $\Lc_j\subset\Pi$ be the straight line $\{z= x+iy: x \in \mathbb{R}, \ y=2^{-j}\}$, and $\m_j$
be the measure $W(j)\d(\Lc_j)$ with some weight sequence $W(j)$. In other words, it is the Lebesgue measure on the line $\Lc_j$ with the weight factor $W(j)$. The sum $\m=\sum_j\m_j$ is a measure on $\Pi$, which, however, does not have compact support in $\Pi$. By Definition \ref{Def Mk}, the measure $\m$ belongs to the class $\MF_{0,\g}$ (say, for $\g=\frac12$), as soon as $W(j)=O(2^{-2j})$.
\end{example}
\begin{example}\label{Ex5} For the same system of straight lines $\Lc_j$ we consider
\begin{equation}\label{Ex5.eq1}
    \varsigma=\sum_j W(j)\varsigma_j\equiv \sum_jW(j) \partial^j\d(\Lc_j)=\sum_j (i/2)^jW(j)(1\otimes \partial^j_y\d(y-2^{-j})).
\end{equation}
In \eqref{Ex5.eq1}, $\varsigma$ is a formal sum of distributions $W(j)\varsigma_j$, each being a derivative of a measure, of unbounded orders, and this formal sum corresponds to the sesquilinear form
\begin{equation}\label{Ex5.eq2}
    \Fb_\varsigma[f,g]=\sum_j\Fb_{\varsigma_j}[f,g]=\sum_j(-1)^jW(j)\int_{\Lc_j}\partial^jf \cdot\bar{g}dx.
\end{equation}
The sequence of weights $W(j)$ should be chosen in such a way that the sum \eqref{Ex5.eq2} converges for $f,g\in \Ac_1$ and, moreover, is a bounded sesquilinear form on $\Ac_1$. By Theorem \ref{Th.concrete}, the measure $\m_j=1\otimes \d(y-2^{-j})$ belongs to the class $\MF_{j,\g}$ with estimate
\begin{equation*}
    \varpi_{j,\g}(\m_j)\le C (2/\g)^{2j} (j!)^2.
\end{equation*}

Thus, if the sum $\sum_j W(j) (2/\g)^{2j} (j!)^2$ is finite, the sesquilinear form \eqref{Ex5.eq2} is bounded on $\Ac_1$.
\end{example}

\section{The structure of the Bergman spaces.}
Along with the Bergman space $\Ac_1$ of analytic functions on $\Pi$, we consider spaces of polyanalytic functions. We denote by $\Ac_j, \, j=1,2,\dots$ the space of square integrable functions on $\Pi$ satisfying the iterated Cauchy-Riemann equation $\overline{\partial}^jf=0$ (the reader was, probably, intrigued by the subscript in the notation $\Ac_1$ -- now its use is justified). Of course, $\Ac_j\subset\Ac_{j'}$ for $j<j'$, so, to get rid of these 'less polyanalytic' functions, the \emph{true polyanalytic} Bergman spaces have been introduced (see \cite{V2}), by
\begin{equation*}
    \Ac_{(j)}=\Ac_j\ominus\Ac_{j-1}= \Ac_j \cap \Ac_{j-1}^{\perp},\ \ j=2,\dots;\ \  \Ac_{(1)}=\Ac_{1}.
\end{equation*}
Worth mentioning is the following direct sum decomposition of $L_2(\Pi)$:
\begin{equation*}
 L_2(\Pi) = \bigoplus_{n \in \mathbb{N}}  \Ac_{(j)} \oplus \bigoplus_{n \in \mathbb{N}}  \widetilde{\Ac}_{(j)},
\end{equation*}
where $\widetilde{\Ac}_{(j)}$ are true poly-antianalytic Bergman spaces (see, for details, \cite{V2}).

 For these poly-Bergman spaces on the upper half-plane, there exists a system of creation and annihilation operators, described in \cite{KarlPessIEOT,V3}. These operators are two-dimensional singular integral operators,
\begin{equation*} 
(\Sbb_\P u)(w)=-\frac1\pi \int_{\P}\frac{u(z)dA(z)}{(z-w)^2} \qquad \text{and} \qquad (\Sbb_\P^*u)(w)=-\frac1\pi \int_{\P}\frac{u(z)dA(z)}{(\bar{z}-\bar{w})^2}.
\end{equation*}
Being understood in the principal value sense, they are bounded in $L^2(\Pi)$ and adjoint to each other. They are, in fact,  the Beurling--Ahlfors operators compressed to the half-plane, and  are surjective isometries,
\begin{eqnarray}\label{SBP}
  \Sbb_\P:  \Ac_{(j)}\to \Ac_{(j+1)}, && \Sbb_\P^*:  {\Ac}_{(j)}\to {\Ac}_{(j-1)}, \quad j>1, \\ \nonumber
  \Sbb_\P^*:  \widetilde{\Ac}_{(j)}\to \widetilde{\Ac}_{(j+1)}, && \Sbb_\P:  \widetilde{\Ac}_{(j)}\to \widetilde{\Ac}_{(j-1)}, \quad j>1,
\end{eqnarray}
while
\begin{equation*}
    \Sbb_\P^*: \Ac_1\to \{0\}, \quad \Sbb_\P: \widetilde{\Ac_1}\to \{0\}.
\end{equation*}

Thus, in particular, we have surjective isometries
\begin{equation*}
  (\Sbb_\P)^j\Ac_{(1)}=\Ac_{(j+1)}(\P).
\end{equation*}

Formulas \eqref{SBP},  possessing the structure similar to the ones of the Landau subspaces for the Schr\"{o}dinger equation with uniform magnetic field,  justify calling  $\Sbb_\P,\, \Sbb^*_\P$ creation and annihilation operators.

\begin{remark} Here one can notice a certain discrepancy in notations: $\Sbb$ denotes usually the \emph{annihilation} operator in the poly-Fock spaces while $\Sbb_\P$ denotes here the \emph{creation} operator in the poly-Bergman spaces -- however, this is the tradition and we do not want to break it.
\end{remark}

The operators $\Sbb_\P^j$, restricted to $\Ac_1$, admit a representation, found in \cite{PessoaPW}, which is much more convenient for using in further reductions.
\begin{theorem}[{\cite[Theorem 3.3]{PessoaPW}}] For $u\in\Ac_1$,
\begin{equation}\label{TransP}
   ( \Sbb_\P^j u)(z)=\frac{{\partial}^{j}[(z-\bar{z})^{j}u(z)]}{j!},  j\ge0.
\end{equation}

\end{theorem}

Note that the isometry  $\Ub$ of the Bergman spaces on the upper half-plane $\Pi$ and on the disk $\mathbb{D}$ is not carried over to the poly-analytic Bergman spaces.

\section{Relations among the Toeplitz operators in the true poly-Bergman spaces.}

Let $a$ be a distribution on $\P$, defined at least on $C^{\infty}(\bar{\Pi})\cap L^1(\Pi).$  We consider the sesquilinear form
\begin{equation}\label{PFform}
 \Fb_a[u,v] =  ( a u,\bar{v})=(a,u\bar{v}), \quad u= \Sbb_\Pi^j f\in\Ac_{(j+1)}, \quad v=\Sbb_\P^j g\in\Ac_{(j+1)},
\end{equation}
with $f,g$ being elements of the standard orthonormal basis in $\Ac_{(1)}$. The sesquilinear form \eqref{PFform} is defined for $f,g$ in the basis in $\Ac_{(1)}$ and can be extended by sesquiliearity to the linear span of the basis. If it turns out that \eqref{PFform} is bounded on this span, it can be extended by continuity to the whole $\Ac_{(1)}$ and thus it would define a bounded operator in $\Ac_{(1)}$. On the other hand, since $\Sbb_\Pi^j$ is a unitary operator,  $\Fb_a[u,v]$ can be therefore extended to a bounded sesquilinear form defined for $u,v$ being arbitrary elements in $\Ac_{(j+1)}$, thus defining a bounded operator in $\Ac_{(j+1)}$. The present section is devoted to finding an explicit relation between these two operators.
Further on, in this section we will only use the Hilbert space $L^2(\Pi,dA)$ and therefore we will suppress the notation of the space in the scalar product $\la .,.\ra;$ the parentheses $(.,.)$ still denote the action of a distribution on a smooth function, without the complex conjugation.  Since $u,v$ are elements in the poly-Bergman space, the product $u\bar{v}$ is a smooth function in $L^1(\Pi)$.

The following result leads to establishing a relation between Toeplitz operators in the true  poly-Bergman space $\Ac_{(j)}$ and the Bergman space $\Ac_1$.

\begin{prop}\label{PropTranshalfplane}
Let $a$ be a distribution in the half-plane $\Pi$. Then
\begin{equation}\label{transPi}
   \Fb_a[u,v] =\Fb_a[\Sbb_\P^j f,\Sbb_\P^j g ]  =( a,\Sbb_\P^j f\overline{\Sbb_\P^j g})=\langle (\Kb_{(j)} a) f,g\rangle,
\end{equation}
with $\Sbb_\P^j$ defined in \eqref{TransP}, where $\Kb_{(j)}$ is a differential operator of order $2j$ having the form
\begin{equation}\label{TRanshalfplane}
    \Kb_{(j)}=\underline{\Kb_{(j)}}(\Delta( y^2\cdot),\bar{\partial}(y\cdot),\partial (y\cdot)),
\end{equation}
and $\underline{\Kb_{(j)}}$ being a polynomial of degree $j.$ Moreover, if we assign the weight $-1$ to the differentiation and the weight $1$ to the multiplication by $y$, with weights adding under the multiplication, then all monomials in $\underline{\Kb_{(j)}}$ have weight $0$.
\end{prop}
\begin{proof}We demonstrate the reasoning for the case $j=1$. The general case uses the same machinery with some tedious bookkeeping.
We set $u=\Sbb_\P f$, $v=\Sbb_\P g$ and consider the sesquilinear form $\Fb_{a}[f,g]=\langle a\partial (y f),\partial (y g)\rangle\equiv (a,\partial (y f)\overline{\partial (y g)} )$ for $f,g$ being some elements in the standard orthonormal basis in $\Ac(\Pi)$. Due to $\partial^*=-\bar{\partial}$:
\begin{gather}\label{transHalfplane2}
    (a,\partial (y f)\overline{\partial (y g)} ) = (a,(-if+ y\partial f)\overline{(-ig+y\partial g)})\\ \nonumber
    =(a,f\bar{g})+(a,-ify\bar{\partial g})+(a,i\bar{g}y\partial f )+(a,y^2(\partial f) {\bar{\partial} \bar{g}})
    \end{gather}
   (the last transformation uses $\bar{\partial}g=0$). For the second term on the right-hand side in \eqref{transHalfplane2}, by the general rules of manipulation with distributions, we have
\begin{equation*}
 (a,-ify\bar{\partial g})= (-iya, f\overline{{\partial g}})=(-iya,\overline{{\partial}(\bar{f}g)})-(-iyF,(\bar{\partial}f)g )
 = (\bar{\partial}(iya),f\bar{g}),
\end{equation*}
because $\bar{\partial}f=0$. The third term on the right in  \eqref{transHalfplane2} is transformed in a similar way, and for the last one,
 \begin{gather*}
   (a,y^2\partial f{\bar{\partial} \bar{g}} )=(y^2a,  \bar{\partial}(\partial f \bar{g})-\bar{\partial}(\partial f)\bar{g} ) =-(\bar{\partial}(y^2 a),(\partial f) \bar{g} )\\ \nonumber
   =-(\bar{\partial}(y^2 a), \partial(f \bar{g})-\{f \partial\bar{g}\})=(\partial\bar{\partial}(y^2 a), f \bar{g}),
  \end{gather*}
   again, the terms in curly bracket vanishing due to  $\bar{\partial}g=0$. Collecting the terms in \eqref{transHalfplane2}, after simple transformations, we obtain the required relation.

For higher order, the procedure of transformation is  similar, by means of formally commuting   $a$ and factors in the creation operators $\Sbb_\P^j$ in the expression $\langle a\Sbb_\P^j f,\Sbb_\P^j g\rangle\equiv(a, \Sbb_\P^j f\overline{\Sbb_\P^j g})$, so that the Cauchy-Riemann operator falls on the functions $f,g$, while any commutation with $a$ produces a derivative of $a$. It remains to notice that when commuting  the terms in the expression on the left-hand side in \eqref{transPi}, the weight of the terms does not change. Alternatively, one can make the calculations similar to the ones shown above, again by moving the Cauchy-Riemann operator to the functions $f,g$ and on $F$.

To make the general reasoning more transparent, we present here  our transformations for the case $j=2$. So, we set $u=\Sbb_{\Pi}^2 f, v=\Sbb_{\Pi}^2 g$.
We start with
\begin{equation}\label{j2.start}
 \Fb_a[u,v]=   \Fb_a[\Sbb_{\Pi}^2 f, \Sbb_{\Pi}^2 g]=(  (\Sbb_{\Pi}^2 f) a,  \overline{\Sbb_{\Pi}^2 g})=-2( (\partial^2 ( y^2 f)\times a), \bar{\partial}^2(y^2\bar{g})).
\end{equation}
We expand in \eqref{j2.start} the derivatives of the product by the Leibnitz formula, to obtain

\begin{equation}\label{j2.1}
    \Fb_a[u,v]=(\partial^2 (( y^2f)a)-2\partial((y^2f)\partial a)+y^2 f\partial^2 a,\bar{ \partial}^2 (y^2 \bar{g}).
\end{equation}
Now we carry over the derivatives $\partial, \partial^2$ to the second factor in \eqref {j2.1} (this is legal due to the definition of the derivatives of distributions):
\begin{gather}\label{j2.2}
\Fb_a[u,v]= ((y^2 f)a, \partial^2(\bar{ \partial}^2 (y^2 \bar{g})))\\\nonumber
+2((y^2f)\partial a, \partial(\bar{ \partial}^2 (y^2 \bar{g})))+
(y^2 f\partial^2 a,\bar{ \partial}^2 (y^2 \bar{g}).
\end{gather}

We consider then the terms in \eqref{j2.2} separately. In the first term, we commute $\partial^2$ and $\bar{\partial}^2$ in the second factor:
\begin{equation*}
     \partial^2(\bar{ \partial}^2 (y^2 \bar{g}))=\bar{\partial}^2({ \partial}^2 (y^2 \bar{g}))=\textstyle{\frac12}\bar{\partial}^2(\bar{g}),
\end{equation*}

since $\partial\bar{g}=0.$ Therefore, the first term in \eqref{j2.2} transforms to
\begin{gather*}
   \textstyle{\frac12} ((y^2 f)a,\bar{\partial}^2 \bar{g})=\textstyle{\frac12}(\bar{\partial}^2(y^2f a),\bar{g})=\textstyle{\frac12}(f \bar{\partial}^2(y^2 F),\bar{g})\\ \nonumber=(f\Kb_1a ,\bar{g})=(\Kb_1a, f\bar{g}),
\end{gather*}
with $\Kb_1a$ being the distribution
\begin{equation*}
\Kb_1a=\textstyle{\frac12}\bar{\partial}^2(y^2 a).
\end{equation*}

 Next, for the second term in \eqref{j2.2}, we have
 \begin{gather*}
2((y^2f)\partial a, \partial(\bar{ \partial}^2 (y^2 \bar{g})))=-2((y^2f)\partial a, \bar{ \partial}^2\partial (y^2 \bar{g})))=\\\nonumber
2((y^2f)\partial a, \bar{ \partial}^2(y \bar{g}))=2(\bar{ \partial}^2(y^2 f \partial a), y\bar{g}).
\end{gather*}
Now,
\begin{equation*}
    \bar{ \partial}^2(y^2 f \partial a)=f\bar{ \partial}^2(y^2\partial a)=f(2\partial a+2y \bar{ \partial}\partial a +y^2 \bar{ \partial}^2\partial a).
\end{equation*}
Thus, the second term in \eqref{j2.2} equals to
\begin{equation*}
    2((y^2f)\partial a, \partial(\bar{ \partial}^2 (y^2 \bar{g})))=(f\Kb_2 a,\bar{g} )=(\Kb_2 a,f\bar{g} ),
\end{equation*}
where
\begin{equation*}
    \Kb_2 a=y(2\partial a+2y \bar{ \partial}\partial a +y^2 \bar{ \partial}^2\partial a).
\end{equation*}
Finally,  the third term in \eqref{j2.2} is transformed as
\begin{gather*}
    (y^2 f\partial^2 a,\bar{ \partial}^2 (y^2 \bar{g})=(\bar{ \partial}^2f y^2 \partial^2 a,y^2 \bar{g})=\\\nonumber (f \bar{\partial}^2(y^2 a), y^2 \bar{g})= (fy^2  \bar{\partial}^2(y^2 a), \bar{g})=(f \Kb_3 a,\bar{g} )=(\Kb_3 a, f\bar{g} ),
\end{gather*}
where $\Kb_3 a= y^2  \bar{\partial}^2(y^2 a).$
A simple bookkeeping shows that the operators $\Kb_1,\Kb_2,\Kb_3$ have the structure claimed by the theorem, $\Kb_{(2)}=\Kb_1+\Kb_2+\Kb_3$.
Note again that although the distribution $a$ does not necessarily have compact support in $\Pi$, our definition of derivatives of such distributions conserves the formal differentiation rules we used in these calculations.
\end{proof}

 As explained above, the equality \eqref{transPi} extends to the whole of $\Ac_1$, as soon as we know that the right-hand side or on the left-hand side is a bounded sesquilinear form in $\Ac_1$. Thus, the statement of Proposition \ref{PropTranshalfplane} can be formulated as the following theorem.

 \begin{theorem}\label{ThEquPi}The operators
 $\Tb_a(\Ac_{(j+1)})$ and $\Tb_{\Kb a}(\Ac_1)$ are unitarily equivalent (up to a numerical factor) as soon as one of them is bounded. In this case, if one of these operators is compact, or belongs to a Schatten class, or  is of finite rank, or zero, then the same holds for the other one.
 \end{theorem}

 The terms in the differential operator $\Kb_{(j)}$ can be regrouped so that it takes the form
 \begin{equation*}
    \Kb_{(j)}=\sum_{p+\bar{p}+2q\le 2j}b_{p,\bar{p},q}(\partial)^p(\bar{\partial})^{\bar{p}}\D^q y^{p+\bar{p}+2q}.
 \end{equation*}

Now we can apply the boundedness conditions obtained earlier, in Section 3, for differential sesquilinear forms to obtain boundedness conditions for Toeplitz operators in true poly-Bergman spaces.
 \begin{theorem}\label{ThBddPoly}Let $\m$ be a measure on $\Pi$ such that $y^k\m$ are $k$-C measures for $\Ac_1$ for $k=0,1,\dots, 2j$.   Then the sesquilinear form $\int f\bar{g}d\m$ is bounded in $\Ac_{(j+1)}$ and defines a bounded Toeplitz operator in $\Ac_{(j+1)}$. If, moreover, $y^k\m$ are vanishing $k$-C measures for $\Ac_1$, then the corresponding operator in $\Ac_{(j+1)}$ is compact.
 \end{theorem}

\end{document}